\documentclass[12pt]{article}
\usepackage{bbm}
\usepackage{graphicx}
\graphicspath{ {./images/} }
\usepackage{wrapfig}
\usepackage[mathscr]{euscript}
\usepackage{subcaption}
\usepackage{hyperref}
\usepackage{epstopdf}
\usepackage{enumerate}
\usepackage{longtable}
\usepackage{float}
\usepackage{amsmath,amsfonts,amssymb,amsthm,epsfig,epstopdf,titling,url,array,multirow}
\usepackage{color}
\usepackage{framed}
\usepackage[utf8]{inputenc}
\usepackage[english]{babel}
\usepackage{xparse}
\usepackage[margin=2.5cm]{geometry}

\NewDocumentCommand{\INTERVALINNARDS}{ m m }{
	#1 {,} #2
}
\NewDocumentCommand{\interval}{ s m >{\SplitArgument{1}{,}}m m o }
{
	\IfBooleanTF{#1}{
		\left#2 \INTERVALINNARDS #3 \right#4
	}{
		\IfValueTF{#5}{
			#5{#2} \INTERVALINNARDS #3 #5{#4}
		}{
			#2 \INTERVALINNARDS #3 #4
		}
	}
}

\newtheorem{thm}{Theorem}[section]
\newtheorem{maintheorem}{Theorem}

\newtheorem{cor}{Corollary}[section]
\newtheorem{lem}{Lemma}[section]
\newtheorem{prop}{Proposition}[section]

\newtheorem{rem}{Remark}[section]
\newtheorem{obs}[thm]{Observation}

\numberwithin{equation}{section}

\date{}
\begin{document}
	\title{Relaxed Newton’s Method as a Family of Root-finding Methods: Dynamics and Convergence}
	\author{Soumen Pal\\Department of Mathematics, Indian Institute of Technology Madras,\\ India\\ Email: soumen.pal.new@gmail.com}
	\date{}
	\maketitle
\begin{abstract}
Relaxed Newton's method is a one-parameter family of root-finding methods that generalizes the classical Newton's method. When viewed as a rational map on the Riemann sphere, this family exhibits rich and subtle global dynamics that depend both on the underlying polynomial and on the relaxation parameter. In this paper, we investigate the complex dynamical behavior of relaxed Newton maps associated with complex polynomials.

We first characterize rational maps that arise as relaxed Newton maps in terms of the multipliers of their fixed points. Our main results identify several explicit classes of polynomials for which relaxed Newton's method is \textit{convergent} for all parameters $h$, in the sense that the Fatou set consists precisely of the basins of attraction of the roots. We further show that this favorable behavior does not hold in general: for any fixed relaxation parameter $h\in \{z:|z-1|<1\}$, there exists a generic cubic polynomial for which the relaxed Newton map fails to be convergent. Additional results include a complete characterization of when the Julia set is a straight line, an analysis of symmetry groups arising from rotational invariance, and sufficient conditions ensuring that all immediate basins of attraction are unbounded.
\end{abstract}
\section{Introduction and Main Results}
Root-finding methods have long been studied from numerical and dynamical
perspectives. When such a method is applied to a complex polynomial, it naturally defines a rational map on the Riemann sphere. The global dynamical properties of this rational map encode detailed information about the convergence behavior of the method, transforming numerical questions into problems in complex dynamics.

The classical Newton's method applied to polynomials is the most prominent example of this correspondence. Its associated Newton map has been studied extensively over the past several decades, leading to deep insights into the structure of Julia sets, the geometry of basins of attraction, and the occurrence of extraneous dynamical phenomena such as attracting cycles not associated with roots; see, for example, \cite{HSS2003,Lei1997,Felex1989}. These investigations have also been extended beyond polynomials to rational and entire functions; see \cite{BFJK2018,NPP2026,RS2007} and the references therein.

Relaxed Newton's method provides a natural one-parameter generalization of the classical Newton's method. Given a polynomial $p$, it is defined by
\begin{equation}\label{Formula_ReNew}
N_{h,p}(z)=z-h\,\frac{p(z)}{p'(z)},
\end{equation}
where $h$ is called the relaxation parameter. From a numerical perspective, this relaxation parameter $h$ is introduced to regulate convergence speed and stability. From a dynamical viewpoint, varying $h$ produces a family of rational maps whose fixed points correspond precisely to the roots of $p$ and the point at infinity, but whose global behavior may differ substantially from that of the classical Newton map ($h=1$).

The dynamics of related higher-order root-finding methods, including K\"onig's, Halley's, and Chebyshev's methods, have attracted considerable attention; see Buff and Henriksen \cite{BH2003}, the recent work of Liu et al. \cite{LPP2026}, and Nayak and Pal \cite{Sym_dyn}. Certain aspects of relaxed Newton's method have also been investigated; see, for example, \c{C}ilingir \cite{Cilingir2006} and Plaza and Romero \cite{PlazaRomero2011}. However, despite its importance in numerical analysis, relaxed Newton's method has not yet been systematically studied from the perspective of global complex dynamics, to the best of our knowledge.

A central problem in the theory of Newton-type methods is that of \emph{convergence}. From the numerical viewpoint, this concerns whether the iterates of a given root-finding method converge to a root for almost every initial point. From the perspective of complex dynamics, convergence corresponds to the situation in which the Fatou set of the associated rational map consists precisely of the basins of attraction of the roots of the polynomial.

For the classical Newton's method, convergence fails in most cases due to the presence of extraneous attracting or parabolic cycles. In the setting of relaxed Newton's method, the introduction of a relaxation parameter $h$ leads to a family of rational maps whose global dynamical behavior may vary significantly with both the polynomial and the parameter.

Rather than seeking convergence for particular values of $h$, the focus of this paper is on identifying classes of polynomials for which relaxed Newton's method is \emph{convergent for all values of the parameter $h$}. In other words, we aim to determine intrinsic algebraic and geometric conditions on the polynomial that guarantee parameter-independent convergence of the associated relaxed Newton maps.

In this paper, we address this question from the viewpoint of complex dynamics.
We first establish a characterization of relaxed Newton maps in terms of the
multipliers of their fixed points, extending ideas that appear implicitly in earlier
works on Newton-type methods. This characterization allows us to identify relaxed
Newton maps among general rational functions with prescribed fixed-point behavior.

Our main contribution is the identification of explicit and natural classes of polynomials for which relaxed Newton's method is \emph{$h$-convergent}, meaning that it is convergent for every choice of the relaxation parameter $h$. These include polynomials with exactly two distinct roots, unicritical polynomials, and families of composite polynomials of the form $(az+b)^m\big((az+b)^n+c\big)$. In each case, the proofs combine critical orbit analysis, symmetry considerations, and classical results from rational dynamics. The precise statement is the following.
\begin{maintheorem}\label{Conv_cls}
For each of the following classes of polynomials, relaxed Newton's method is $h$-convergent:
\begin{enumerate}[(i)]
	\item Polynomials having exactly two roots;
	\item Unicritical polynomials;
	\item Polynomials of the form  $(a z+b)^{m}\left[(a z+b)^{n}+c\right]$, where $a, b, c \in \mathbb{C}$, $a, c \neq 0$ and $m, n \in \mathbb{N}$ with $n\geq 2$. 
\end{enumerate}
\end{maintheorem}
A root-finding method is said to be \textit{generally convergent} for the space of polynomials $\mathcal{P}_d$ of a fixed degree $d\geq 2$ if the method is convergent for every polynomial belonging to a dense subset of $\mathcal{P}_d$. The classical Newton's method is generally convergent for the space of quadratic polynomials. Theorem~\ref{Conv_cls} shows that relaxed Newton's method is $h$-convergent for all quadratic polynomials; in other words, it is generally convergent, independently of the choice of $h$. For a cubic polynomial $p$, the same conclusion holds whenever $p$ has exactly two roots, is unicritical, or is of the form $p(z)=(a z+b)\left[(a z+b)^{2}+c\right]$ where $a, c \neq 0$. However, this behavior does not persist in general. The following result establishes this fact.
\begin{maintheorem}\label{Cubic_nconv}
For any given $h\in \mathbb{D}$, there is a generic, non-unicritical polynomial $p$ such that $N_{h,p}$ is not convergent.
\end{maintheorem}
For the classical Newton's method, the Julia set is a line if and only if a polynomial having exactly two roots with equal multiplicity is considered. We extend this result for relaxed Newton's method.
\begin{maintheorem}\label{not_line}
	Let $p$ be a polynomial with exactly two distinct roots.
	Then the Julia set of relaxed Newton's method $N_{h,p}$ is a line if and only if the two roots have equal multiplicity and $h$ is real.
\end{maintheorem}
There may exist Euclidean isometries that can preserve the Julia set of a rational function $R$. The collection of such isometries is called the \textit{symmetry group} of $R$ and is denoted by $\Sigma R$. Therefore,
$$\Sigma R=\{\sigma(z)=az+b: |a|=1\text{ and }\sigma(\mathcal{J}(R))=\mathcal{J}(R)\}.$$
We study the symmetry group of relaxed Newton's method applied to polynomials having $n$th order rotational symmetry. We establish the equality in the symmetry groups of the polynomials and that of the relaxed Newton maps.
\begin{maintheorem}\label{Same_sym}
Consider the polynomial $p(z)=z^m(z^n-1)$, where $m\geq 0$ and $n\geq 2$. If $\mathcal{J}(N_{h,p})$ is not a line, then $\Sigma p= \Sigma N_{h,p}$, for $h\in \{z:|z-1|<1\}$.
\end{maintheorem}
To the best of our knowledge, this work provides the first systematic study of relaxed Newton's method as a family of rational maps on the Riemann sphere, addressing convergence, geometric finiteness, symmetry, and parameter-independent dynamical behavior. In this sense, the present paper bridges a gap between numerical root-finding theory and modern complex dynamics.

The structure of the paper is as follows. In Sect.~\ref{Sec_B&P} we collect background results that will be used throughout the article. Sect.~\ref{Sec_RNM} is devoted to geometric and dynamical properties of relaxed Newton maps, including fixed-point classification, the Scaling property, and the characterization of relaxed Newton maps among rational functions. The proofs of the first two main theorems are presented in Sect.~\ref{Sec_CC}. Finally, Sect.~\ref{Sec_SDR} contains further dynamical results, including symmetry considerations and the proofs of the remaining main theorems.
\section{Background and Preliminary Discussions}\label{Sec_B&P}

Let $R:\widehat{\mathbb{C}}\to \widehat{\mathbb{C}}$ be a rational function of degree at least two, where $\widehat{\mathbb{C}}$ denotes the Riemann sphere. The \textit{Fatou set} $\mathcal{F}(R)$ of $R$ is the set of all points $z\in \widehat{\mathbb{C}}$ for which the family of iterates $\{R^n\}_{n\geq 1}$ is normal (in the sense of Montel) in a neighborhood of $z$. The complement $\mathcal{J}(R)=\widehat{\mathbb{C}}\setminus \mathcal{F}(R)$ is called the \textit{Julia set} of $R$.

The Fatou set $\mathcal{F}(R)$ is open and completely invariant under $R$, while the Julia set $\mathcal{J}(R)$ is closed, non-empty, and completely invariant. Moreover, the Julia set is the locus of chaotic dynamics, whereas the dynamics on the Fatou set is stable. Each connected component of $\mathcal{F}(R)$ is called a \textit{Fatou component}, and the boundary of every Fatou component is contained in the Julia set.

For a rational function $R$, the \textit{multiplier} (we denote it by $\lambda$) of a fixed point $z_0\in \mathbb{C}$ is defined as $R'(z_0)$. If $\infty$ is a fixed point of $R$, then its multiplier is defined as $S'(0)$, where $S(z)=\frac{1}{R(\frac{1}{z})}$. Depending on the multiplier, a fixed point is classified as \textit{attracting} if $|\lambda|<1$ (\textit{superattracting} if $\lambda=0$), \textit{repelling} if $|\lambda|>1$, and \textit{indifferent} if $|\lambda|=1$. A fixed point is called \textit{weakly repelling} if it is either repelling or has multiplier equal to $1$.

The \textit{residue fixed point index} of $R$ at a fixed point $z_0$ is defined as 
$$
\iota (R,z_0)= \frac{1}{2 \pi i} {\oint\limits_{\gamma}\frac{1}{z-R(z)}dz},
$$
where $\gamma$ is a small positively oriented closed curve around $z_0$ that does not surround any other fixed point of $R$. If $\lambda\neq 1$, then the residue fixed point index satisfies
\begin{equation}\label{FPI_simple}
\iota (R,z_0)=\frac{1}{1-\lambda}.
\end{equation}
The sum of the residue fixed point indices over all fixed points of a rational function equals $1$ (see \cite[Theorem 12.4]{Milnor_book}).

A fixed point of a rational function $R$ is said to be multiple if it is a root of $R(z)-z=0$ with multiplicity at least two. Equivalently, $z_0$ is a multiple fixed point if and only if $R'(z_0)=1$. Buff and Henriksen prove the following.
\begin{lem}{\cite[Lemma 7]{BH2003}}\label{equal_rat}
	Let $R$ and $S$ be two rational functions having only simple fixed points. If both functions share the same fixed points with the same respective multipliers, then they are identical.
\end{lem}

The next result, proved by Shishikura, provides a condition for the connected Julia set of a rational function.
\begin{thm}{\cite[Theorem I]{Shishikura2009}}\label{shishikura}
If the Julia set of a rational function $R$ of degree at least two is disconnected, then there exist at least two weakly repelling fixed points lying on two different components of the Julia set. In particular, if $R$ has exactly one repelling fixed point, then $\mathcal{J}(R)$ is connected.
\end{thm}

The symmetry group of the Julia set of a polynomial is well studied. A polynomial is said to be normalized if it is monic and its second leading coefficient is $0$. Such a polynomial $p$ can be represented as $p(z)=z^\alpha p_0(z^\beta)$, where $p_0$ is a monic polynomial, and $\alpha$, $\beta$ are maximal for the expression. Beardon proved the following.
\begin{thm}{\cite[Theorem 9.5.4]{Beardon_book}}\label{poly-symm}
	For a normalized polynomial $p$,
	$\Sigma p=\{\sigma: \sigma(z)=\lambda z, \lambda^\beta=1\}.$ 
\end{thm}

A relation between the symmetry group of a normalized polynomial $p$ and that of a root-finding method applied to $p$ is established in \cite{Sym_dyn}.
\begin{thm}{\cite[Theorem 1.1]{Sym_dyn}}\label{rel_sym}
	If a root-finding method $F$ satisfies the Scaling property, then for every normalized polynomial $p$, $\Sigma p\subseteq \Sigma F_p$.
\end{thm}

For an attracting fixed point $z_0$, the set $\mathcal{B}_{z_0}=\{z:\lim\limits_{n\to \infty}R^n(z)=z_0\}$ is called \textit{basin} of $z_0$. A connected component of $\mathcal{B}_{z_0}$ that contains $z_0$ is called the \textit{immediate basin}, denoted by $\mathcal{A}_{z_0}$.
\section{Relaxed Newton's Method: Formula, Properties and Characterization}\label{Sec_RNM}
Recall (\ref{Formula_ReNew}), for a polynomial $p$, the corresponding relaxed Newton map is defined as
$$N_{h, p}(z)=z-h \frac{p(z)}{p'(z)}.$$ 
If $p$ is linear or a monomial, then $N_{h, p}$ is linear. From now onwards, to study relaxed Newton maps, we consider polynomials having at least two distinct roots.
\subsection{Fixed Points and Properties}
It is clear from (\ref{Formula_ReNew}) that a finite fixed point of $N_{h,p}$ is precisely a root of $p$. Now
$$
N_{h, p}^{\prime}(z)=1-h+h \frac{p(z) p^{\prime \prime}(z)}{\left(p^{\prime}(z)\right)^{2}}. 
$$
It follows that a simple root of $p$ is a fixed point of $N_{h, p}$ with multiplier $1-h$, whereas if $z_{0}$ is a root of $p$ with multiplicity $m_{0}$, then its multiplier as a fixed point of $N_{h,p}$ is
$$
\lambda_{z_{0}}=N_{h, p}^{\prime}\left(z_{0}\right) =1-h+h \frac{m_{0}-1}{m_{0}}=1-\frac{h}{m_{0}}.$$
Since the relaxed Newton's method represents a family of root-finding methods, a root $z_{0}$ of $p$ with multiplicity $m_{0}$ should be an attracting fixed point of $N_{h, p}$. Thus, we get  $\left|1-\frac{h}{m_{0}}\right|<1$, and this gives $\left|h-m_{0}\right|<m_{0}$.
Therefore 
$$h \in \mathbb{D}_{m}(m)=\{z:|z-m|<m\},$$ 
where $m$ is the least multiplicity among all the roots of $p$. More precisely, if $p$ has a simple root, then $h \in \mathbb{D}_{1}(1)$.
\par 
The point at $\infty$ is a fixed point of $N_{h,p}$ with multiplier $\frac{d}{d-h}$, where $d=\deg(p)$. Since $p$ has at least two distinct roots, we have $m<d$, where $m$ is the minimum of the multiplicities of the roots of $p$. Consequently, $|d-h|<d$, and hence, $\infty$ is a repelling fixed point of $N_{h,p}$. Therefore, $N_{h,p}$ has exactly one repelling fixed point, while all other fixed points are attracting. By Theorem~\ref{shishikura}, it follows that the Julia set of $N_{h,p}$ is connected.

Like the classical Newton's method, relaxed Newton's method satisfies the Scaling property.
\begin{lem}[Scaling property]
For every polynomial $p$ and every affine map $T, N_{h, g}(z)=T^{-1}\circ N_{h, p}\circ T(z)$, where $g(z)=\lambda (p\circ T)(z)$ and $\lambda \in \mathbb{C} \backslash\{0\}$.
\end{lem}
\begin{proof}
Let $T(z)=A z+B, A \neq 0$.
Then $g'(z)=\lambda A p'(A z+B)$ and
$$N_{h, g}(z)=z-\frac{p(A z+B)}{A p'(A z+B)}.$$
Therefore,
$$A N_{h, g}(z)+B =(A z+B)-\frac{p(A z+B)}{p^{\prime}(A z+B)} =N_{h, p}(A z+B),$$
which implies that
$$T\left(N_{h, g}(z)\right) =N_{h, p}(T(z)).$$
Hence the result.
\end{proof}
\begin{obs}[Usefulness of the Scaling property]\label{Use_SP}
Using the Scaling property we can consider a polynomial to be monic whenever we study a root-finding method that satisfies the Scaling property. Moreover, this property allows specific polynomial families to be represented by a single polynomial. Consequently, studying the root-finding method for the representative polynomial suffices for the entire family. The following are some instances.
\begin{enumerate}
\item Polynomials having exactly two roots: Suppose $\alpha$ and $\beta$ are two roots of the polynomial $p$. Then by pre-composing $p$ with the affine map $T(z)=\frac{1}{2}((\alpha -\beta)z+(\alpha +\beta)$, we can take the roots to be $1$ and $-1$. 
\item  Unicritical polynomial: A unicritical polynomial $(z-\alpha)^n+\beta$, where $n\geq 2$ and $\beta \neq 0 $ (we exclude the case whenever the polynomial is linear or a monomial), can be taken as $z^n-1$ by pre-composing the affine map $T(z)=Az+\alpha$, where $A^n=-\beta$, and by multiplying the constant $\lambda=-\frac{1}{\beta}$.
\item Non-unicritical cubic polynomial: Consider a cubic polynomial $p(z)=z^3+a_1z^2+a_2z+a_3$. It is not unicritical whenever $a_1^2\neq 3a_2$. Then for some $A\in \mathbb{C}$, where $A^2=\frac{1}{9}(a_1^2-3a_2)$ and $\xi= -\frac{a_1}{3}$, the polynomial $\frac{1}{A^3}p(Az+\xi)$ is of the form $z^3-3z+\frac{p(\xi)}{A^3}$. Thus, to study the relaxed Newton's method applied to non-unicritical cubic polynomials, it is sufficient to consider the one-parameter family of cubic polynomials $z^3-3z+a$, where $a\in \mathbb{C}$. 
\end{enumerate}  
\end{obs}
\subsection{Characterization}
\begin{lem}\label{equal_pwr}
For every polynomial $p$ and every $n\in \mathbb{N}$, $N_{h, p}=N_{nh,p^n}$.
\end{lem}
\begin{proof}
Let $\alpha$ be a root of $p$ with multiplicity $m$. Then it is a root of $p^n$ with multiplicity $mn$. Therefore, $\alpha$ is a fixed point of both $N_{h, p}$ and $N_{nh,p^n}$ with the same multiplier $1-\frac{h}{m}$. Thus, by Lemma~\ref{equal_rat}, we get $N_{h, p}=N_{nh,p^n}$.
\end{proof}
\begin{cor}
If all the finite fixed points of $N_{h,p}$ are superattracting, then all the roots of $p$ have the same multiplicity. Moreover, if $m$ is the multiplicity of all roots of $p$, i.e., $p=q^m$, where $q$ is a generic polynomial, then $N_{h,p}=N_q$.
\end{cor}
As all the fixed points of $N_{h,p}$ are simple, and the distinct roots of $p$ are the only finite fixed points of $N_{h,p}$, the degree of $N_{h,p}$ is exactly the number of distinct roots of $p$. We consider quadratic relaxed Newton maps. In this case $p$ is having exactly two distinct roots, and using the Scaling property, we consider
$$
p(z)=(z-1)^{k}(z+1)^{m},
$$
where $k, m \in \mathbb{N}$. Then the corresponding relaxed Newton map is
\begin{align*}
N_{h, p}(z) & =z-h \frac{z^{2}-1}{(k+m) z+(k-m)}  \nonumber\\
& =\frac{(k+m-h) z^{2}+(k-m) z+h}{(k+m) z+(k-m)}. \nonumber
\end{align*}
\begin{lem}[Characterization of quadratic relaxed Newton maps]
Let $R$ be a quadratic rational function having two attracting and one repelling fixed point. Then the following are true.
\begin{enumerate}[(i)]
\item If the multipliers of two attracting fixed points are the same, say $\lambda$, then $R$ is conjugate to $N_{h, p}$, where
$
h=1-\lambda \text { and } p(z)=z^{2}-1
$.
\item If one of the fixed points of $R$ is superattracting and the multiplier of another attracting fixed point is $\frac{n}{m}$, where $m\in \mathbb{N}$ and $n \in \mathbb{N}\setminus \{0\}$, then $R$ is conjugate to $N_{h, p}$, where $h=m-n$ and $p(z)=(z-1)^{m-n}(z+1)^{m}$.
\item If the ratio of the residues of the fixed point indices corresponding to the attracting fixed points of $R$ is $\frac{k}{m}$, where $k,m\in \mathbb{N}$, then there exists $h\in \mathbb{C}\setminus \{0\}$ such that $R$ is conjugate to $N_{h, p}$, where $p(z)=(z-1)^{k}(z+1)^{m}$.
 \end{enumerate}
\end{lem}
\begin{proof}
Let $R$ be a quadratic rational function having three distinct fixed points, among them, one being repelling and the other two being attracting. Using conjugacy, we can consider $ 1 $ and $ -1 $ to be attracting, whereas $\infty$ is the repelling fixed point of $R$. Thus, $R$ has the following form:
\begin{equation}\label{form_R}
R(z)=z-\frac{z^{2}-1}{A z+B},
\end{equation}
where $A, B \in \mathbb{C}$. Then
\begin{equation}\label{mult_R}
R^{\prime}(1)=1-\frac{2}{A+B} \text { and } R^{\prime}(-1)=1-\frac{2}{A-B}.
\end{equation}
\begin{enumerate}
	\item If $R^{\prime}(1)=R^{\prime}(-1)=\lambda$ then $B=0$ and $\lambda=1-\frac{2}{A}$.
Therefore from (\ref{form_R}), we have
$$R(z) =z-\frac{z^{2}-1}{A z}=z-\frac{2}{A} \frac{z^{2}-1}{2 z} \\
=N_{h, p}(z),$$
where $p(z)=\left(z^{2}-1\right)$ and $h=\frac{2}{A}=1-\lambda$.
\item  Let $R^{\prime}(1)=0$
and $R^{\prime}(-1)=\frac{n}{m}$. Then we get $A+B=2$ and $A-B=\frac{2 m}{m-n}$.
Therefore, $A=\frac{2 m-n}{m-n}$ and $B=-\frac{n}{m-n}$.
Hence from (\ref{form_R}),
$$
\begin{aligned}
R(z) & =z-\frac{(m-n)\left(z^{2}-1\right)}{(2 m-n) z-n} \\
& =z-(m-n) \frac{z^{2}-1}{((m-n)+m) z+((m-n)-m)} =N_{h, p}(z),
\end{aligned}
$$
where $h=m-n$ and $p(z)=(z-1)^{m-n}(z+1)^{m}$.
\item  From (\ref{FPI_simple}), we have $\iota(1)=\frac{1}{1-R'(1)}$ and $\iota(-1)=\frac{1}{1-R'(-1)}$. Thus $\frac{\iota(1)}{\iota(-1)}=\frac{1-R^{\prime}(-1)}{1-R^{\prime}(1)}=\frac{k}{m}$. Following from part two of this proof, without loss of generality we can assume that $1<k<m$. Now it follows from (\ref{mult_R}) that
$\frac{A+B}{A-B}=\frac{k}{m}.$
Therefore, there exists $c \in \mathbb{C} \setminus \{0\}$ such that $A+B=2 k c$ and $A-B=2 m c$. Thus we get $A=(m+k) c$ and $B=(m-k) c$. This gives that $$R(z)=z-\frac{1}{c} \frac{z^{2}-1}{(k+m) z+(k-m)}=N_{h, p},$$
 where $h=\frac{1}{c}$ and $p(z)=(z-1)^{k}(z+1)^{m}$.
\end{enumerate}
This concludes the proof.
\end{proof}
Here is a generalized version.
\begin{lem}(Characterization of relaxed Newton's method)
Let $R$ be a rational function of degree $d\geq 2$ whose all but one of the fixed points are attracting and the other one is repelling. If there is a non-zero complex number $h$ such that the multiplier of every attracting fixed point $\alpha_i$, $i=1,2,\dots, d$, can be written as $\frac{m_i-h}{m_i}$ for some $m_i\in \mathbb{N}$, then for a suitable choice of a M\"{o}bius map $\varphi$, the rational function $R$ is the same as $N_{h,p}$, where $p(z)=\prod_{i=1}^{d}(z-\varphi(\alpha_i))^{m_i}$.
\end{lem}
\begin{proof}
If $\infty$ is the repelling fixed point of $R$, then consider $p(z)=\prod_{i=1}^{d}(z-\alpha_i)^{m_i}$. Therefore, $\alpha_i$ is a fixed point of $N_{h,p}$ with multiplier $1-\frac{h}{m_i}$. Therefore, the fixed points of $R$ and their respective multipliers are the same as that of $N_{h,p}$. Hence by Lemma~\ref{equal_rat}, we get $R=N_{h,p}$.
\par
Let $\beta\in \mathbb{C}$ be the repelling fixed point of $R$. We can choose a M\"{o}bius map $\varphi$ that takes $\beta$ to $\infty$ and then consider $S=\varphi\circ R\circ \varphi^{-1}$. The finite fixed points of $S$ are $\varphi(\alpha_i)$ with multipliers $\frac{m_i-h}{m_i}$. Consider $p(z)=\prod_{i=1}^{d}(z-\alpha_i)^{m_i}.$ Then from the above first part of this proof we get that $S=N_{h,p}$ and the proof is completed.
\end{proof}
This result can be reformulated as follows. Let $R$ be a rational function of degree $d\geq 2$ having exactly one repelling fixed point, all other fixed points are attracting. Suppose that the residue fixed point indices of $R$ at all but one of the fixed points are of the form $\frac{m_i}{h}$, where $m_i\in \mathbb{N}$ for $i=1,2,\dots, d$, and $h\in \{z:|z-m|<m\}$, where $m=\min\{m_i: i=1,2,\dots, d\}$. Then $R$ is conjugate to $N_{h,p}$ for some polynomial $p$.
\section{$h$-Convergent Classes: Proofs of Theorems~\ref{Conv_cls} \& \ref{Cubic_nconv}}\label{Sec_CC}
Since $\infty$ is the only repelling fixed point of $N_{h, p}$, by Theorem~\ref{shishikura} we get that $\mathcal{J}\left(N_{h, p}\right)$ is connected. However, for some polynomial $p$, $N_{h,p}$ can have non-repelling cycles, resulting in the successive iterations of some points not converging to a root of $p$. Our intention is to find classes of polynomials where this is not going to be true. In other words, we present some classes of polynomials where relaxed Newton's method is convergent. There are three such classes of polynomials in Theorem~\ref{Conv_cls}. We prove the theorem for these classes in the first three subsections. In the last subsection, we prove Theorem~\ref{Cubic_nconv}.
\subsection{Polynomials having exactly two roots:}
In this case, using the Scaling property, we can consider $p(z)=(z-1)^{k}(z+1)^{m}$, where $k\leq m$.
Here $h \in \mathbb{D}_{k}(k)$.
Note that $\deg(N_{h, p})=2$, and therefore, $N_{h, p}$ has exactly two critical points. As $-1$ and $1$ are attracting fixed points, their respective immediate basins $\mathcal{A}_{\pm 1}$ contain a critical point. Therefore, there are no attracting or parabolic cycles of $N_{h,p}$, and the Fatou set consists of the basins of $\pm 1$. Moreover, by the Riemann-Hurwitz formula we get that the immediate basins are completely invariant, and hence the Julia set is a Jordan curve.
\subsection{$p$ is uncritical:}
In this case, using the Scaling property, we consider $p(z)=z^{n}-1, n \geq 3$. Then
$p'(z)=n z^{n-1}$ and $ p''(z)=n(n-1) z^{n-2}$
so that $$N_{h, p}(z)=z-h \frac{z^{n}-1}{n z^{n-1}}=\frac{(n-h) z^{n}+h}{n z^{n-1}} \text{ and } N_{h, p}'(z)=\frac{(n-h) z^{n}-h(n-1)}{n z^{n}}.$$
Note that, by Theorem~\ref{rel_sym}, we have $\left\{z \mapsto \lambda z: \lambda^{n}=1\right\} \subseteq \Sigma N_{h, p}.$ The origin is a critical point of $N_{h, p}$ with multiplicity $(n-2)$. A non-zero critical point of $N_{h, p}$ is the solution of
$$
z^{n}=\frac{h(n-1)}{n-h}.
$$
As $\infty$ is a repelling fixed point, $\infty \in \mathcal{J}\left(N_{h, p}\right)$ and therefore $0 \in \mathcal{J}\left(N_{h, p}\right)$. All non-zero critical points are in the respective immediate basins corresponding to the roots of $p$. Therefore, $\mathcal{F}(N_{h,p})$ consists of the basins of roots of $p$.
\subsection{$p(z)=(a z+b)^{m}\left[(a z+b)^{n}+c\right]$ where $a, b, c \in \mathbb{C}$, $a, c \neq 0$ and $m, n \in \mathbb{N}$ with $n\geq 2$:}
Using the Scaling property, we can consider $a=1, b=0$, and $c=-1$. Thus we have $p(z)=z^{m}\left(z^{n}-1\right)$ and $
p'(z) =z^{m-1}\left[(m+n) z^{n}-m\right]
$.
Relaxed Newton's method applied to $p$ gives
$$N(z)= z-h\frac{ z(z^{n}-1)}{(m+n) z^{n}-m}= \frac{(m+n-h) z^{n+1}-(m-h) z}{(m+n) z^{n}-m}.$$
$$
\begin{aligned}
N^{\prime}(z)=&  {\left[((n+1)(m+n-h) z^{n}-(m-h))((m+n) z^{n}-m)\right.} \\
& \left.-((m+n-h) z^{n+1}-(m-h) z)(n(m+n) z^{n-1})\right] /((m+n) z^{n}-m)^{2} \\
= & {\left[(m+n)(m+n-h) z^{2 n}-(m(n+1)(m+n-h)+(m+n)(m-h)\right.} \\
& \left.-n(m+n)(m-h)) z^{n}+m(m-h)\right] /((m+n) z^{n}-m)^{2}.
\end{aligned}
$$
Thus the critical points of $N$ are the solutions of
\begin{equation}\label{crt_pt}
A z^{2 n}-B z^{n}+C=0,
\end{equation}
where
$A=(m+n)(m+n-h)$, $ B=2\left(m^{2}+m n-m h\right)+n^2h-n h $ and $ C=m(m-h). $ From (\ref{crt_pt}), we get that there are two sets of critical points, $S_{1}=\left\{z: z^{n}=\frac{B + \sqrt{B^{2}-4 A C}}{2A}\right\}$ and $S_{2}=\left\{z: z^{n}=\frac{B - \sqrt{B^{2}-4 A C}}{2A}\right\}$. Note that every element of each set is preserved under rotations about the origin of order $n$. Moreover, the following are true.
\begin{obs}\label{Prop_ReNew}
\begin{enumerate}
\item The Julia set of $N$ is preserved under rotations about the origin of order $n$, i.e., $\left\{z \mapsto \lambda z: \lambda^{n}=1\right\} \subseteq \Sigma N$.
\item Let $\mathcal{A}_{0}$ be the immediate basin of $0$. Then for each $\sigma \in \left\{z \mapsto \lambda z: \lambda^{n}=1\right\}$, we have $\sigma\left(\mathcal{A}_{0}\right)=\mathcal{A}_{0}$.
\item  If $\mathcal{A}_{1}, \mathcal{A}_{2}, \dots, \mathcal{A}_{n}$ are the immediate basins corresponding to the roots of $z^{n}=1$, then $\sigma\left(\mathcal{A}_{i}\right)=\mathcal{A}_{j}$ for $ i, j \in\{1,2, \dots, n\}$ and $ i \neq j$. Thus, $\left\{\mathcal{A}_{i}: i=1,2, \ldots, n\right\}$ is preserved under rotations about the origin of order $n$.
\end{enumerate}
\end{obs}
Since each immediate basin contains at least one critical point, $\bigcup\limits_{i=1}^{n} \mathcal{A}_{i}$ contains all critical points that are preserved under $n$th order rotations about the origin.
Thus, either $S_{1}$ or $S_{2}$ is in $\bigcup\limits_{i=1}^{n} \mathcal{A}_{i}$. Without loss of generality, let $S_{1} \subseteq \bigcup\limits_{i=1}^{n} \mathcal{A}_{i}$. This implies that $\mathcal{A}_{0}$ contains a critical point that belongs to $S_{2}$. From Observation~\ref{Prop_ReNew}(2) we have $S_{2} \subseteq \mathcal{A}_{0}$. Therefore, we get that the Fatou set
$\mathcal{F}(N)$ consists of the basins corresponding to the roots of $p$ and hence, $N$ is generally convergent for the family $(a z+b)^{m}\left[(a z+b)^{n}+c\right]$.

\begin{figure}[h!]
	\begin{subfigure}{.6\textwidth}
		\centering
		\includegraphics[width=0.78\linewidth]{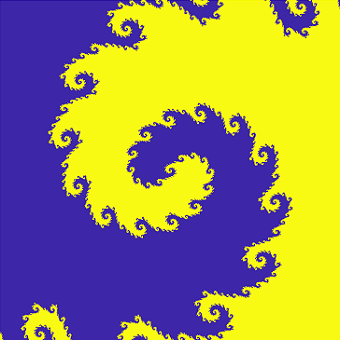}
		\caption{ $p(z)=z^2-1$, $h=(2+\pi i)/4$}
	\end{subfigure}
	\hspace{-2.0cm}
	\begin{subfigure}{.6\textwidth}
		\centering
		\includegraphics[width=0.78\linewidth]{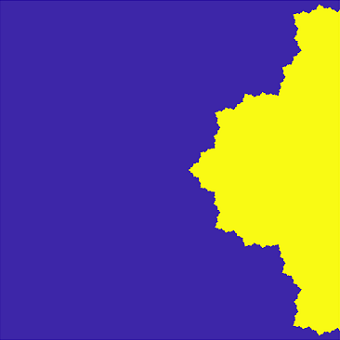}
		\caption{$p(z)=(z-1)(z+1)^2$, $h=1.5$}
	\end{subfigure}
	\caption{Relaxed Newton's method applied to polynomials with exactly two roots: The Julia set is the boundary of any two differently colored regions.}
	\label{JS-exact2}
\end{figure}
\begin{figure}[h!]
	\begin{subfigure}{.6\textwidth}
		\centering
		\includegraphics[width=0.78\linewidth]{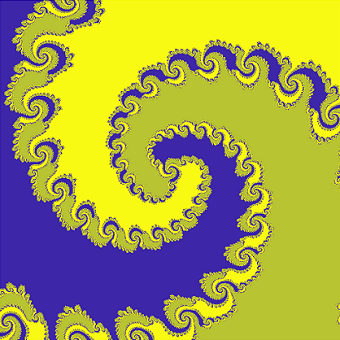}
		\caption{ $p(z)=z^3-1$, $h=(2+\pi i)/4$}
	\end{subfigure}
	\hspace{-2.0cm}
	\begin{subfigure}{.6\textwidth}
		\centering
		\includegraphics[width=0.78\linewidth]{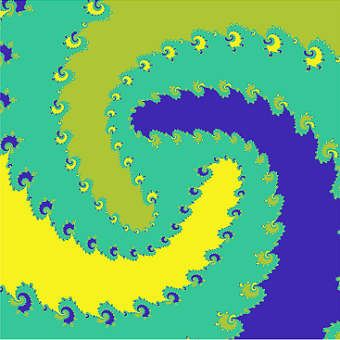}
		\caption{$p(z)=z(z^3-1)$, $h=(2+\pi i)/4$}
	\end{subfigure}
	\caption{Relaxed Newton's method applied to $z^3-1$ and $z(z^3-1)$: The Fatou set $\mathcal{F}(N_{h,p})$ is the union of basins corresponding to the roots of the respective polynomials}
	\label{JS-uni_ori}
\end{figure}
\begin{rem}\label{lcJS}
In all the above cases, we see that the resulting relaxed Newton's method is a geometrically finite map with connected Julia set. Thus, the Julia set is locally connected.
\end{rem}
Figs.~\ref{JS-exact2} and \ref{JS-uni_ori} show the dynamical plane of the relaxed Newton maps $N_{h,p}$, where each of the polynomials is a member of the $h$-convergent classes defined in
Theorem~\ref{Conv_cls}(1--3). 
\subsection{Proof of Theorem~\ref{Cubic_nconv}}
From Observation~\ref{Use_SP}(3) we consider a generic, non-unicritical polynomial as $p(z)=z^3-3z+a$, where $a\neq \pm 2$. Then 
$$N_{h,p}(z)=\frac{(3-h)z^3-3(1-h)z-ah}{3(z^2-1)},$$
$$N'_{h,p}(z)=\frac{(3-h)z^4-6z^2+2ahz+3(1-h)}{3(z^2-1)^2},$$ and
$$N^2_{h,p}(z)
=\frac{
	(3-h)\,\Phi(z)^3
	-27(1-h)\Phi(z)(z^2-1)^2
	-27ah\,(z^2-1)^3
}{
	9(z^2-1)\bigl(\Phi(z)^2-9(z^2-1)^2\bigr)
},$$
where $\Phi(z)=(3-h)z^3-3(1-h)z-ah.$

Therefore, a critical point $\xi$ of $N_{h,p}$ follows the equation
\begin{equation}\label{Crit_N}
(3-h)\xi^4-6\xi^2+2ah\xi+3(1-h)=0.
\end{equation}
We want to construct a $2$-periodic superattracting cycle of $N_{h,p}$ by computing $N^2_{h,p}(\xi)=\xi$. Note that a $2$-periodic point must be a solution of $N^2_{h,p}(z)-z$, which gives\\
$$
\begin{aligned}
&h(z^3-3z+a)((h^3-12h^2+45h-54)z^6-(6h^3-54h^2+153h-162)z^4+\\
&(2ah^3-15ah^2+18ah)z^3+(9h^3-54h^2+135h-162)z^2-\\
&(6ah^3-27ah^2+18ah)z+(a^2h^3-3a^2h^2-27h+54))=0.
\end{aligned} 
$$
As $\xi$ cannot be a root of $p$ (unless it would be a fixed point of $N_{h,p}$), we have
\begin{align}\label{2-p_sol}
&(h^3-12h^2+45h-54)\xi^6-(6h^3-54h^2+153h-162)\xi^4\nonumber\\&+(2ah^3-15ah^2+18ah)\xi^3+ (9h^3-54h^2+135h-162)\xi^2-\nonumber\\&(6ah^3-27ah^2+18ah)\xi+(a^2h^3-3a^2h^2-27h+54)=0
\end{align}
Now from (\ref{crt_pt}) we have
\begin{equation}\label{val_a}
a=\frac{(h-3)\xi^4+6\xi^2+3(h-1)}{2h\xi}.
\end{equation}
Thus, putting this value in (\ref{2-p_sol}), we get
$$9(h-3)(\xi^2-1)^3((h^2-8h+13)\xi^2-(h-1)^2=0.$$
Observe that, if $\xi^2=1$, then $a=\pm 2$, which is not possible. Thus $\xi =\pm \frac{h-1}{\sqrt{h^2-8h+13}}$, where the principal branch of the square root $\sqrt{h^2-8h+13}$ is considered. Consequently, $$a=\pm \frac{2(h^4-12h^3+57h^2-127h+108)}{h(h^2-8h+13)\sqrt{h^2-8h+13}}.$$
Thus, for any given $h$ there is a non-unicritical generic cubic polynomial whose relaxed Newton's method is not convergent. This completes the proof.

Consider $h=0.5$. Then for the polynomial $p(z)=z^3-3z+1834/(37\sqrt{37})$, the relaxed Newton map $N_{0.5,p}$ has a $2$-periodic superattracting cycle, namely $\{\frac{1}{\sqrt{37}},1.4765\}$. In Fig.~\ref{JS-notConv}, red color represents the basins of this cycle.
\begin{figure}[h!]
	\begin{subfigure}{.6\textwidth}
		\centering
		\includegraphics[width=0.754\linewidth]{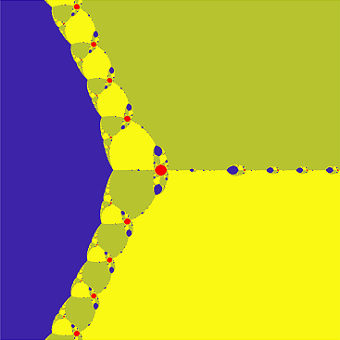}
		\caption{ $p(z)=z^3-3z+1834/(37\sqrt{37})$, $h=0.5$}
	\end{subfigure}
	\hspace{-2.0cm}
	\begin{subfigure}{.6\textwidth}
		\centering
		\includegraphics[width=0.806\linewidth]{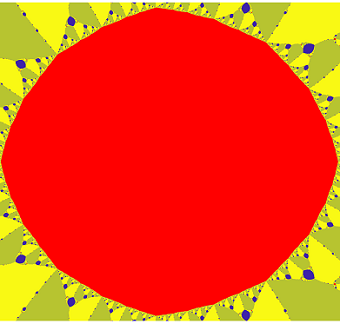}
		\caption{The immediate basin of $1/\sqrt{37}$}
	\end{subfigure}
	\caption{Non-convergent relaxed Newton's method: Red color represents the region where successive iterations of points do not converge to any of the roots of the polynomial}
	\label{JS-notConv}
\end{figure}
\section{Some Dynamical Remarks}\label{Sec_SDR}
\begin{proof}[Proof of Theorem~\ref{not_line}]
By Observation~\ref{Use_SP}(1), we may assume that $1$ and $-1$ are the two roots of $p$ with multiplicities $k$ and $m$, respectively, that is,
$$p(z)=(z-1)^k(z+1)^m.$$
	
We first show that $\mathcal{J}(N_{h,p})$ can be a line only if $k=m$.
The relaxed Newton map associated with $p$ is given by
$$N_{h,p}(z)= z - \frac{h(z-1)(z+1)}{(k+m)z+(k-m)}= \frac{(k+m-h)z^2+(k-m)z+h}{(k+m)z+(k-m)}.	$$
	
Consider the Möbius map
$$\varphi(z)=\frac{z+1}{z-1},$$
and define
$$S=\varphi\circ N_{h,p}\circ \varphi^{-1}.$$
By Theorem~\ref{Conv_cls}(1), $\mathcal{J}(N_{h,p})$ is a Jordan curve, and therefore $\mathcal{F}(N_{h,p})$ is the union of two completely invariant immediate basins, namely $\mathcal{A}_{\pm1}$. Note that $0$ and $\infty$ are attracting fixed points of $S$ with multipliers $1-\frac{h}{m}$ and $1-\frac{h}{k}$, respectively, while $1$ is a repelling fixed point of $S$ with multiplier $\frac{k+m}{k+m-h}$. Moreover, by \cite[Theorem~3.1.4]{Beardon_book}, we have
$\mathcal{J}(S)=\varphi\bigl(\mathcal{J}(N_{h,p})\bigr).$ Thus, $\mathcal{J}(S)$ is a Jordan curve, and $\mathcal{F}(S)$ contains exactly two components, immediate basins of $0$ and $\infty$, respectively.
	
Since a M\"obius map sends a line to either a line or a circle, and since $\infty$ is an attracting fixed point of $S$, the Julia set $\mathcal{J}(S)$ cannot be a line. Hence, if $\mathcal{J}(N_{h,p})$ is a line, then $\mathcal{J}(S)$ must be a circle. Since $0$ is an attracting fixed point of $S$, by K{\oe}nigs Linearization
Theorem (see \cite[Theorem 8.2]{Milnor_book}) there exists a conformal change of coordinates in a neighborhood
of $0$ which conjugates $S$ to the linear map $z\mapsto (1-\frac{h}{m}) z$. Therefore, the immediate basin of $0$ is locally invariant under radial contraction, and its boundary, namely $\mathcal{J}(S)$,
must also be invariant. This is possible only if $\mathcal{J}(S)$ is a
circle centered at $0$.
	
Note that $\frac{m-k}{m+k}$
is a pole of $N_{h,p}$. Moreover,
$$\varphi\!\left(\frac{m-k}{m+k}\right)=\frac{m-k+m+k}{m-k-m-k}=-\frac{m}{k}.$$
Thus, if $\mathcal{J}(S)$ is a circle, then its radius must be $\frac{m}{k}$.
	
Again, $\infty\in\mathcal{J}(N_{h,p})$ and $\varphi(\infty)=1.$
This implies that $1\in\mathcal{J}(S)$, which is possible only when $m=k$.
	
Now assume that $p$ has exactly two roots of equal multiplicity. By Lemma~\ref{equal_pwr}, it suffices to consider $p(z)=z^2-1.$
The relaxed Newton map associated with $p$ is given by
$$N_{h,p}(z)=\frac{(2-h)z^2+h}{2z}.$$
Note that $0$ is a pole of $N_{h,p}$ and hence belongs to $\mathcal{J}(N_{h,p})$. 
	
Again, consider the M\"obius map $\varphi(z)=\frac{z+1}{z-1}.$
A direct computation shows that
$$
\varphi\circ N_{h,p}\circ \varphi^{-1}(z)= z\left(\frac{z+a}{1+az}\right),\quad \text{where } a=1-h.$$
Hence, $\mathcal{J}(N_{h,p})$ is a line if and only if
$\mathcal{J}(\varphi\circ N_{h,p}\circ \varphi^{-1})$
is the unit circle.
Again, the unit circle is invariant under the M\"obius map
$z\mapsto \frac{z+a}{1+az}$ if and only if $a\in\mathbb{R}$, which is equivalent to $h\in\mathbb{R}$. This completes the proof.
\end{proof}
Fig.~\ref{JS-exact2} shows that $\mathcal{J}(N_{h,p})$ is not a line, although polynomials having exactly two roots are considered. In Fig.~\ref{JS-exact2}(a), we consider $m=n=1$ and $h$ is non-real; whereas, in Fig.~\ref{JS-exact2}(b), $h$ is taken to be real, but the roots of $p$ are with different multiplicities.

Another important aspect is whether the immediate basins of $N_{h,p}$ corresponding to the roots of $p$ are unbounded. This is trivial for the classical Newton's method. The proof relies on the fact that all multipliers of the fixed points of $N_{p}$ are real (see \cite[Proposition 6]{HSS2003} and \cite[Remark 3]{Felex1989}). We establish a general result on the unboundedness of invariant Fatou components for an arbitrary rational function having exactly one repelling fixed point at infinity.
\begin{lem}[Unboundedness of invariant immediate basins]
Let $R:\widehat{\mathbb C}\to\widehat{\mathbb C}$ be a rational map of degree $d\ge2$ such that:
\begin{enumerate}
\item $\infty$ is the only repelling fixed point of $R$;
\item every finite fixed point of $R$ is attracting (possibly superattracting);
\item either
 \begin{enumerate}[(i)]
	\item the Julia set $\mathcal{J}(R)$ is locally connected, or
	\item the argument of the multiplier of every attracting fixed point of $R$ is rational.
 \end{enumerate}
\end{enumerate}
Then every invariant Fatou component of $R$ is unbounded.
\end{lem}
\begin{proof}
Let $U$ be an invariant Fatou component of $R$.
Then $U$ contains a fixed point $z_0$ of $R$, which is attracting by hypothesis.
Hence $U$ is the immediate basin of attraction of $z_0$.
	
By K{\oe}nigs linearization, there exists a conformal isomorphism $\phi:U\to\mathbb D$
such that $\phi\circ R = \lambda\,\phi,$
where $\lambda=R'(z_0)$ and $|\lambda|<1$. For $\theta\in\mathbb R$, define the internal ray
$$\gamma_\theta=\phi^{-1}\{re^{i\theta}:0<r<1\}.$$
	
If $\arg(\lambda) \in 2\pi t$, where $t \in \mathbb Q$, (or if $z_0$ is superattracting), there exists at least one angle $\theta$ for which $\gamma_\theta$ is fixed by $R$.
Under the hypothesis 3\emph{(i)}, Carath\'eodory's theorem implies that $\phi$ extends continuously to $\partial U$, and hence every internal ray lands at a unique point of $\partial U$.
Under the hypothesis 3\emph{(ii)}, the fixed internal ray also lands at a unique point in $\partial U$. This follows from the standard ray-landing argument for rational angles (see, for example, \cite[Theorem 18.10]{Milnor_book}), which extends directly to immediate basins of rational maps.
	
Let $z^*\in\partial U$ be the landing point of such a fixed internal ray.
Then $R(z^*)=z^*$.
Since $z^*\in \mathcal{J}(R)$, the fixed point $z^*$ cannot be attracting.
By hypothesis, $\infty$ is the only repelling fixed point of $R$, and there are no parabolic fixed points.
Thus $z^*=\infty$.
	
Since a fixed internal ray is unbounded and lands at $\infty$, its immediate basin $U$ must be unbounded.
Therefore, every invariant Fatou component of $R$ is unbounded.
\end{proof}

\begin{rem}
The proof relies only on the existence and landing of at least one fixed internal ray.
Thus the hypotheses may be weakened by replacing \emph{3}(i) or \emph{3}(ii) with any condition guaranteeing unique landing of internal rays, such as local connectivity of $\partial U$ or semi-hyperbolicity.
\end{rem}	
In the case of the relaxed Newton map $N_{h,p}$ associated with a polynomial $p$, the multiplier at a finite fixed point depends explicitly on the relaxation parameter $h$.
In particular, it is sufficient to assume that $\arg(h)\in\mathbb Q$ to ensure the existence of a fixed internal ray, and hence, the unboundedness of invariant immediate basins follows from the above argument.
\begin{prop}\label{unbdd_imm}
If $\mathcal{J}(N_{h,p})$ is locally connected, or $\arg(h)\in\mathbb{Q}$, then all immediate basins of $N_{h,p}$ corresponding to the roots of $p$ are unbounded.
\end{prop}	
If the Julia set is not locally connected and $\arg(h)\notin\mathbb Q$, then fixed internal rays need not exist or internal rays may fail to land.
In this setting, ray-based arguments are no longer available, and the unboundedness of immediate basins, if it holds, must be proved using different, more global methods (possibly  using quasi-conformal surgery). 

For $h=1$, the rigidity of Newton's method has been extensively studied, and the Julia set has been shown to be locally connected, with a few exceptions (see~\cite{DS2022}). The same question arises here: Under what conditions is $\mathcal{J}(N_{h,p})$ locally connected? We believe that this is true in general.

Since relaxed Newton's method satisfies the Scaling property, it follows from Theorem~\ref{rel_sym} that $\Sigma p\subseteq\Sigma N_{h,p}$, where $p$ is a normalized polynomial. Now consider a polynomial whose symmetry group contains at least one non-identity element. The natural question is when these two sets coincide. From the previous result, we see that for certain values of $h$, the Julia set $N_{h,p}$ is a line, and in such cases equality does not hold. However, an affirmative answer can be obtained once we exclude the situations in which $\mathcal{J}(N_{h,p})$ is a line.
\begin{proof}[Proof of Theorem~\ref{Same_sym}]
	Since relaxed Newton's method satisfies the Scaling property, it follows from Theorems~\ref{poly-symm} and \ref{rel_sym} that $\Sigma p=\{z\mapsto \mu z: \mu^n=1\}\subseteq\Sigma N_{h,p}$. For $m=0$, $n=2$, and real $h$, the Julia set of $N_{h,p}$ is a line. Following the hypothesis, we exclude this case. As $n\geq 2$, every element in $\Sigma N_{h,p}$ is a rotation about the origin only. Note that $N_{h,p}$ is convergent and its Julia set is locally connected (see Theorem~\ref{Conv_cls} and Remark~\ref{lcJS}). Therefore, by Proposition~\ref{unbdd_imm}, we get that every immediate basin of $N_{h,p}$ corresponding to the roots of $p$ is unbounded. Now, following \cite[Lemma 3.3]{Sym_dyn}, we conclude that these immediate basins are the only unbounded Fatou components of $N_{h,p}$.  Moreover, these are preserved under the rotations about the origin of order $n$. Therefore, $\Sigma N_{h,p}$ can contain at most $n$ elements, and hence, the symmetry groups of $p$ and $N_{h,p}$ are identical.
\end{proof}
Theorem~\ref{Same_sym} suggests a possible rigidity phenomenon relating the symmetry group of a polynomial and that of its relaxed Newton map. Motivated by this result, we propose the following conjectural direction.

Let $p$ be a normalized polynomial whose symmetry group $\Sigma p$ is non-trivial. We expect that for all parameters $h$,
$$
\Sigma p=\Sigma N_{h,p}
$$
holds whenever the Julia set $\mathcal{J}(N_{h,p})$ is not a line.

A possible approach toward this problem is through the analysis of unbounded Fatou components of $N_{h,p}$. In particular, it is natural to investigate whether the presence of a non-repelling periodic cycle of $N_{h,p}$ can lead to the existence of an unbounded Fatou component. If every unbounded Fatou component of $N_{h,p}$ is necessarily associated with an attracting fixed point corresponding to a root of $p$, then the global symmetry of relaxed Newton's method would be forced to coincide with that of the underlying polynomial.

This leads to the following natural question: \textit{Can the relaxed Newton map $N_{h,p}$ admit a non-repelling periodic cycle whose basin of attraction contains an unbounded Fatou component?} A negative answer would provide a crucial step toward establishing the conjectured equality of symmetry groups in greater generality.
\subsection*{Acknowledgements}The author acknowledges the support of the Indian Institute of Technology Madras through a Postdoctoral Fellowship (No.F.Acad./IPDF/R12/2025). He sincerely thanks his mentor, Prof. Saminathan Ponnusamy, for his careful revision of the manuscript.
\section{Declarations}
\subsection{Funding} Not applicable.

\subsection{Conflict of Interests}
The author declares that he has no conflict of interest regarding the publication of this paper.

\subsection{Data Availability Statement}
The author declares that this research is purely theoretical and is not associated with any data.

\subsection{Code availability} Not applicable.

\end{document}